\documentclass[english,12pt,a4paper]{amsart}
\usepackage[T1]{fontenc}
\usepackage{lmodern}
\usepackage[width=14cm, height=21cm]{geometry}
\usepackage{graphicx}
\usepackage{mathtools}
\usepackage{amssymb}
\usepackage{amsthm}
\usepackage{amsmath}
\usepackage{babel}
\usepackage{hyperref}
\usepackage{enumerate}

\theoremstyle{plain}
\newtheorem{theorem}{Theorem}[section]
\newtheorem{corollary}[theorem]{Corollary}
\newtheorem{lemma}{Lemma}[section]

\begin{document}
	\title[On Ramanujan-G\"ollnitz-Gordon  continued fraction]{On  new identities connecting Ramanujan-G\"ollnitz-Gordon  continued fraction and Ramanujan's continued fraction of order four}
	
\author{Shruthi C. Bhat}
\address{Shruthi C. Bhat, Manipal Institute of Technology, Manipal Academy of Higher Education, Manipal-576104, Karnataka, India.}
\email{shruthi.research1@gmail.com}

\author{B. R. Srivatsa Kumar}
\address{B. R. Srivatsa Kumar, Manipal Institute of Technology, Manipal Academy of Higher Education, Manipal-576104, Karnataka, India.}
\email{srivatsa.kumar@manipal.edu}

\maketitle
		
		\begin{abstract}
			By employing the classical tools from the theory of $q$-series and theta functions, new fascinating identities on different continued fractions can be achieved. In this article, we use the product expansion of Jacobi's theta function to establish identities that connect Ramanujan-G\"ollnitz-Gordon  continued fraction with Ramanujan's continued fraction of order four. Also, we obtain Lambert series identities using Ramanujan's $_1 \psi_1$ summation formula.  
		\end{abstract}
		
		
		
		\section{Introduction}
		Classical theta functions provide a strong framework for studying Ramanujan's continued fractions, which play key role in the theory of special functions. Ramanujan introduced several continued fractions which have deep connections to $q$-series,  Eisenstein series, and modular equations. A detailed description of Ramanujan's continued fractions and $q$-series identities can be found in his notebooks \cite{Ramanujan1957, Ramanujan2000} and in extensive commentaries of Berndt \cite{Berndt1991}. In the recent years, many mathematicians found the theory of continued fractions intriguing and worth revisiting. 
		
		The Entry 12 \cite[p. 24]{Berndt1991} in Ramanujan's second notebook gives the following general form of continued fraction. Let $k,l,q \in \mathbb{C}$ with $|kl|<1$ and $|q|<1$ or that $k=l^{2m+1}$ for some $m\in \mathbb{Z}$. Then
		
		\begin{align}
			\dfrac{(k^2q^3;q^4)_{\infty}(l^2q^3;q^4)_{\infty}}{(k^2q;q^4)_{\infty}(l^2q;q^4)_{\infty}} = \dfrac{1}{(1-kl)+\dfrac{(k-lq)(l-kq)}{(1-kl)(q^2+1)+\dfrac{(k-lq^3)(l-kq^3)}{(1-kl)(q^4+1)+\ldots}}}.\label{gcf}
		\end{align}
		The above continued fraction can also be found in \cite{Adiga_1985}. Ramanujan \cite{Ramanujan1957,Ramanujan1988} dealt with various special cases of the above expression and recorded algebraic relations on them.	The following is a special case of general continued fraction, known as Ramanujan-G\"ollnitz-Gordon  continued fraction.   
		\begin{align}\label{H}
			\mathfrak{H}(q):=q^{1/2} \dfrac{f(-q,-q^7)}{f(-q^3,-q^5)}=\dfrac{q^{1/2}}{1+q+\dfrac{q^2}{1+q^3+\dfrac{q^4}{1+q^5+\ldots}}},
		\end{align}
		where	
		\begin{align}\label{id1}
			f(\gamma,\delta)=\sum_{j=-\infty}^{\infty} \gamma^{j(j+1)/2}\delta^{j(j-1)/2},  \quad |\gamma \delta|<1,
		\end{align} 
		is the well-known Ramanujan's general theta function \cite[p. 35, Entry 19]{Berndt1991}. Ramanujan recorded the above product representation of $\mathfrak{H}(q)$ in his notebook \cite[p. 229]{Ramanujan1957}. G\"ollnitz \cite{Gollnitz1967} and Gordon \cite{Gordon1965} rediscovered and established the same product representation, being unaware of Ramanujan's work. Later, Andrews \cite{Andrews1968} proved \eqref{H} as corollary of a more general result. Ramanujan \cite[p. 229]{Ramanujan1957} also documented the following identities associated to $\mathfrak{H}(q)$, proof of which can be found in Berndt's book\cite[p. 221]{Berndt1991}:
		\begin{align}\label{H-H}
			\dfrac{1}{\mathfrak{H}(q)} - \mathfrak{H}(q) = \dfrac{\varphi(q^2)}{q^{1/2} \psi(q^4)}
		\end{align} and 
		\begin{align}\label{H+H}
			\dfrac{1}{\mathfrak{H}(q)} + \mathfrak{H}(q) = \dfrac{\varphi(q)}{q^{1/2} \psi(q^4)},
		\end{align}
		where 
		\begin{align*}
			\varphi(q):=f(q, q)=\frac{(-q;-q)_\infty}{(q;-q)_\infty}= \sum_{j=-\infty}^{\infty} q^{j^2}
		\end{align*} and 
		\begin{align*}
			\psi(q) := f(q,q^3)= \frac{(q^2;q^2)_\infty}{(q;q^2)_\infty}=\sum_{j=0}^{\infty} q^{j(j+1)/2}.
		\end{align*}
		
		A lot of wonderful work on Ramanujan-G\"ollnitz-Gordon continued fraction can be found in the literature. Chan and Huang \cite{Chan_1997} established relations between $\mathfrak{H}(q)$ and  $\mathfrak{H}(q^n)$ using modular equations of degree $n$. Further, they obtained the explicit evaluations as well.
		Vasuki and Srivatsa Kumar \cite{Vasuki_2006} obtained three new identities connecting $\mathfrak{H}(q)$ with $\mathfrak{H}(q^5)$, $\mathfrak{H}(q^7)$ and $\mathfrak{H}(q^{11})$. Moreover, during this process, they provided a whole new simple approach to arrive at the relation connecting $\mathfrak{H}(q)$ and $\mathfrak{H}(q^3)$ which was previously established by Chan and Huang using a different method. Recently, Chaudhary and Vanitha \cite{Chaudhary_2025} established several Eisenstein series identities associated to $\mathfrak{H}(q)$ and explored the relations to combinatorial partition identities.
		
		Naika \textit{et al.} \cite{Naika_2011} dealt with the continued fraction of order four which is defined as follows.
		\begin{align}\label{I}
			\mathfrak{I}(q):= \dfrac{f(-q,-q^3)}{f(-q^2,-q^2)}=\dfrac{1}{1+\dfrac{q}{1+\dfrac{q^2+q}{1+\dfrac{q^3}{1+\dfrac{q^4+q^2}{1+\ldots}}}}}.
		\end{align}
		They established results on $\mathfrak{I}(q)$ those are analogous to Rogers-Ramanujan and cubic continued fractions.
		
		Adiga \textit{et al.} \cite{Adiga_2014} proved identities connecting Ramanujan's cubic continued fraction with a continued fraction of of order six. Further, they established the associated identities as well. Also, for similar work on continued fractions of order twelve and sixteen, \cite{Adiga_2014b} and \cite{Vanitha_2025} can be referred respectively.
		
		On employing the product expansion of Jacobi's theta function,  the identities relating continued fractions of different orders can be derived. The purpose of this article is to obtain a  family of new identities involving Ramanujan-G\"ollnitz-Gordon continued fraction. The evaluation of Jacobi's theta function $\theta_1$ at special arguments $z=\dfrac{\pi}{8}, \dfrac{2\pi}{8}, \dfrac{3\pi}{8}$ and utilization of their product representations yield the explicit expressions for the infinite products 
		\begin{align*}
			\Gamma_k(q) = \prod_{n=1}^{\infty} (1+\beta_kq^n +q^{2n} ), \qquad 1\leq k\leq 3,
		\end{align*} where the coefficients $\beta_k$ are algebraic numbers arising from fourth root of unity. These products play an important role in relating theta functions to Ramanujan's continued fractions. In Section 2, we provide all the definitions and preliminary results required. In Section 3, we give the identities that involve Ramanujan-G\"ollnitz-Gordon continued fraction and continued fraction of order four. In Section 4, we establish new Lambert series identities associated to Ramanujan-G\"ollnitz-Gordon continued fraction. 
		
		\section{Definitions and Preliminary results}
		Throughout the article, let $q \in \mathbb{C}$, such that $ |q| < 1$. Then $q$-Pochhammer symbol or $q$-shifted factorial \cite{Berndt1991} is defined as follows. For any $\delta \in \mathbb{C}$ and $m \in \mathbb{Z^+}$,
		\begin{align*}
			(\delta;q)_0:=1, \quad		(\delta;q)_m:= \prod_{j=1}^{m}(1-\delta q^{j-1}) \quad \text{and} \quad 	(\delta;q)_\infty:= \prod_{j=0}^{\infty}(1-\delta q^{j}).
		\end{align*}
		\noindent For each $\delta \in \mathbb{C}$, the above product is uniformly and absolutely convergent in every compact subset of the unit disc $|q| < 1$. We use the following notation for ease.
		\begin{align*}
			(\delta_1, \delta_2, \ldots, \delta_m;q)_\infty = (\delta_1;q)_\infty (\delta_2;q)_\infty \ldots (\delta_m;q)_\infty.
		\end{align*}
		The Jacobi's triple product identity states that for $z\ne 0$, 
		\begin{align*}
			\sum_{j=-\infty}^{\infty} (-1)^j q^{\frac{j(j-1)}{2}} z^j = (q;q)_\infty (z;q)_\infty (q/z;q)_\infty.
		\end{align*}
		After Ramanujan,
		\begin{align*}
			f(\beta,\gamma) = (-\beta;\beta \gamma)_\infty (-\gamma;\beta \gamma)_\infty (\beta \gamma;\beta \gamma)_\infty, \qquad |\beta \gamma|<1.
		\end{align*}
		Dedekind-eta function is defined as
		\begin{align}\label{de}
			\eta(\tau) := q^{1/24} (q;q)_\infty,
		\end{align} where $q=e^{2\pi i \tau}$, $Im \tau >0$.
		A Lambert series is a series of the form
		\begin{align*}
			L(q) = \sum_{n=1}^{\infty} a_n \dfrac{q^n}{1-q^n}, \qquad |q|<1,
		\end{align*} 
		where $\{a_n\}_{n\geq 1}$ is an arithmetic sequence. Expanding
		\begin{align*}
			\frac{1}{1-q^n}=\sum_{m=0}^{\infty} q^{mn},
		\end{align*}
		one obtains
		\[
		\sum_{n=1}^{\infty} \frac{a_n q^n}{1-q^n}
		=
		\sum_{N=1}^{\infty}
		\left(
		\sum_{d\mid N} a_d
		\right) q^N,
		\]
		showing that Lambert series naturally generate divisor sums.
		The following results are useful in proving the main results. 
		\begin{lemma}[\cite{Adiga_1985}]
			We have 
			\begin{align}
				f(\gamma,\gamma \delta^2) f(\delta,\gamma^2\delta) &= f(\gamma,\delta) \psi(\gamma \delta), \label{f1}
			\end{align}
			\begin{align} 
				f(\gamma,\delta) +f(-\gamma,-\delta) &= 2f(\gamma^3\delta,\gamma \delta^3), \label{f2}
			\end{align}	and 
			\begin{align}
				f(\gamma,\delta) -f(-\gamma,-\delta) &= 2\gamma f\left(\frac{\delta}{\gamma},\gamma^5 \delta^3\right). \label{f3}
			\end{align}
		\end{lemma}
		Subtracting \eqref{f3} from \eqref{f2}, one can deduce
		\begin{align}
			f(-\gamma,-\delta) =f(\gamma^3\delta,\gamma \delta^3) -\gamma f\left(\frac{\delta}{\gamma}, \gamma^5 \delta^3\right). \label{f4}
		\end{align}

		\section{Identities connecting Ramanujan-G\"ollnitz-Gordon continued fraction and Ramanujan's continued fraction of order four}
		Jacobi's theta function, $\theta_1$, is defined as 
		\begin{align}\label{t1}
			\nonumber	\theta_1(z|\tau) &= 2 \sum_{j=0}^{\infty} (-1)^j q^{\frac{(2j+1)^2}{8}} \sin (2j+1) z\\
			&= 2 q^{1/8} \sum_{j=0}^{\infty} (-1)^j q^{\frac{j(j+1)}{2}} \sin (2j+1) z.
		\end{align}
		\noindent 	In \cite{Chan_2009}, Chan \textit{et al.} showed that 
		\begin{align}\label{t2}
			\nonumber	2  \sum_{j=0}^{\infty} (-1)^j q^{\frac{j(j+1)}{2}} \sin (2j+1) z &= \sum_{j=- \infty}^{\infty} (-1)^j q^{\frac{j(j+1)}{2}} \sin (2j+1) z\\
			&= 2 \sin z (q;q)_\infty (qe^{2i z};q)_\infty (qe^{-2i z};q)_\infty
		\end{align}
		Combining the above two equations, one can see that 
		\begin{align}\label{t3}
			\nonumber	\theta_1(z|\tau) &= 2 q^{1/8} \sin z (q;q)_\infty (qe^{2i z};q)_\infty (qe^{-2i z};q)_\infty \\
			&= i q^{1/8} e^{-iz} (q;q)_\infty (e^{2iz};q)_\infty (qe^{-2iz};q)_\infty.
		\end{align}
		Substituting $z=\frac{\pi}{8}$, $\frac{2\pi}{8}$, and $\frac{3\pi}{8}$ respectively in \eqref{t3}, one can deduce the following:
		\begin{align}\label{t81}
			\theta_1\left(\frac{ \pi}{8} | \tau\right) = 2q^{1/8} \left(\sin\frac{\pi}{8}\right) (q;q)_\infty (qe^{\frac{2\pi i}{8}};q)_\infty (qe^{\frac{-2\pi i}{8}};q)_\infty ,
		\end{align} 
		\begin{align}\label{t82}
			\theta_1\left(\frac{ 2\pi}{8} | \tau\right) = 2q^{1/8} \left(\sin\frac{2\pi}{8}\right) (q;q)_\infty (qe^{\frac{4\pi i}{8}};q)_\infty (qe^{\frac{-4\pi i}{8}};q)_\infty ,
		\end{align} and 
		\begin{align}\label{t83}
			\theta_1\left(\frac{ 3\pi}{8} | \tau\right) = 2q^{1/8} \left(\sin\frac{3\pi}{8}\right) (q;q)_\infty (qe^{\frac{6\pi i}{8}};q)_\infty (qe^{\frac{-6\pi i}{8}};q)_\infty.
		\end{align} 
		Multiplying all the three equations, upon using the identities
		\begin{align}\label{prodsine}
			\sin\left(\dfrac{ \pi}{8}\right) \sin\left(\dfrac{2 \pi}{8}\right) \sin\left(\dfrac{ 3\pi}{8}\right)= \dfrac{1}{4}
		\end{align}
		and
		\begin{align*}
			(1-y) (1-ye^{\frac{8\pi i}{8}}) &(1-ye^{\frac{2\pi i}{8}}) (1-ye^{\frac{-2\pi i}{8}}) (1-ye^{\frac{4\pi i}{8}}) (1-ye^{\frac{-4\pi i}{8}})\\
			& \times (1-ye^{\frac{6\pi i}{8}}) (1-ye^{\frac{-6\pi i}{8}})= (1-y^{8}),
		\end{align*} one can deduce the following:
		\begin{align}\label{tm}
			\theta_1\left(\frac{\pi }{8}|\tau\right) \theta_1\left(\frac{2\pi }{8}|\tau\right) \theta_1\left(\frac{3\pi }{8}|\tau\right)  = \dfrac{2  \eta^3(\tau) \eta(8\tau)}{\eta(2\tau)}.
		\end{align}
		Let
		\begin{align*}
			&\beta_1 := -2 \cos \left(\dfrac{2\pi}{8}\right) = - \sqrt{2} \qquad \beta_2 := -2 \cos \left(\dfrac{4\pi}{8}\right) = 0,\\ \text{and }
			&\beta_3 := -2 \cos \left(\dfrac{6\pi}{8}\right) = \sqrt{2}.
		\end{align*}
		Using the above and the definition of $\eta(\tau)$, \eqref{t81}-\eqref{t83} can be rewritten as 
		\begin{align}\label{O1}
			\Gamma_1(q) := \prod_{n=1}^{\infty} (1+\beta_1q^n +q^{2n} ) = q^{-\frac{1}{12}} \dfrac{\theta_1(\frac{ \pi}{8}|\tau)}{2\eta(\tau)(\sin\frac{\pi}{8})},
		\end{align} 
		\begin{align}\label{O2}
			\Gamma_2(q) := \prod_{n=1}^{\infty} (1+\beta_2q^n +q^{2n} ) = q^{-\frac{1}{12}} \dfrac{\theta_1(\frac{ 2\pi}{8}|\tau)}{2\eta(\tau)(\sin\frac{2\pi}{8})},
		\end{align} and 
		\begin{align}\label{O3}
			\Gamma_3(q) := \prod_{n=1}^{\infty} (1+\beta_3q^n +q^{2n} ) = q^{-\frac{1}{12}} \dfrac{\theta_1(\frac{3 \pi}{8}|\tau)}{2\eta(\tau)(\sin\frac{3\pi}{8})}.
		\end{align}
		Multiplying the above three identities, using the identities \eqref{prodsine} and \eqref{tm}, the following identity is obtained.
		\begin{align}\label{prodK}
			\Gamma_1(q) \Gamma_2(q) \Gamma_3(q)  = q^{-\frac{1}{4}} \dfrac{\eta(8\tau)}{\eta(2\tau)}.
		\end{align} 
		With these, we give our main results here.
		
		\begin{theorem}
			Let $\Gamma_1$  and $\Gamma_3$ be defined as in \eqref{O1} and\eqref{O3} respectively, where $\beta_1 = - \sqrt{2}$ and $\beta_3 = \sqrt{2}$. Then, we have the following: 
			\begin{enumerate}[(i)]
				\item
				\begin{align}\label{O1-O3}
					\nonumber	\Gamma_1^2(q^{1/2}) & \Gamma_3(q^{1/2}) - 	\Gamma_1(q^{1/2})  \Gamma_3^2(q^{1/2}) \\ &=  \dfrac{2\sqrt{2}q^{-3/32} \eta^{1/2}(4\tau) \eta^{1/4}(\tau) \eta(8\tau)}{\eta(\tau/2) \eta^{3/4}(2\tau)} \sqrt{\mathfrak{H}(q)} \sqrt[4]{\mathfrak{I}(q)}.
				\end{align}
				\item 
				\begin{align}\label{O1+O3}
					\nonumber	\Gamma_1^2(q^{1/2}) & \Gamma_3(q^{1/2}) + 	\Gamma_1(q^{1/2}) \Gamma_3^2(q^{1/2}) \\ &= \dfrac{2q^{-3/32} \eta^{1/2}(4\tau) \eta^{1/4}(\tau) \eta(8\tau)}{\eta(\tau/2) \eta^{3/4}(2\tau)} \left(\dfrac{\sqrt[4]{\mathfrak{I}(q)}}{\sqrt{\mathfrak{H}(q)}}-\sqrt{\mathfrak{H}(q)}\sqrt[4]{\mathfrak{I}(q)}\right).
				\end{align}
				\item 	
				\begin{align}\label{33O3-11O1}
					\nonumber (1+\beta_3)	\Gamma_1(q^{1/2}) & \Gamma_3^2(q^{1/2}) - (1+\beta_1)	\Gamma_1^2(q^{1/2})  \Gamma_3(q^{1/2}) \\ &= \dfrac{2\sqrt{2}q^{-3/32} \eta^{1/2}(4\tau) \eta^{1/4}(\tau) \eta(8\tau)}{\eta(\tau/2) \eta^{3/4}(2\tau)}\dfrac{\sqrt[4]{\mathfrak{I}(q)}}{\sqrt{\mathfrak{H}(q)}}.
				\end{align}
				\item 	
				\begin{align}\label{11O1+33O3}
					\nonumber (1+\beta_1)	\Gamma_1^2(q^{1/2}) \Gamma_3(q^{1/2}) &+ (1+\beta_3)	\Gamma_1(q^{1/2})  \Gamma_3^2(q^{1/2}) \\ &=  \dfrac{2q^{-3/32} \eta^{1/2}(4\tau) \eta^{1/4}(\tau) \eta(8\tau)}{\eta(\tau/2) \eta^{3/4}(2\tau)} \left(\dfrac{\sqrt[4]{\mathfrak{I}(q)}}{\sqrt{\mathfrak{H}(q)}}+\sqrt{\mathfrak{H}(q)}\sqrt[4]{\mathfrak{I}(q)}\right).
				\end{align}
				\item 	
				\begin{align}\label{3O3-1O1}
					\nonumber \beta_3 &	\Gamma_1(q^{1/2})  \Gamma_3^2(q^{1/2}) - \beta_1	\Gamma_1^2(q^{1/2})  \Gamma_3(q^{1/2}) \\ &=  \dfrac{2\sqrt{2}q^{-3/32} \eta^{1/2}(4\tau) \eta^{1/4}(\tau) \eta(8\tau)}{\eta(\tau/2) \eta^{3/4}(2\tau)} \left(\dfrac{\sqrt[4]{\mathfrak{I}(q)}}{\sqrt{\mathfrak{H}(q)}}-\sqrt{\mathfrak{H}(q)}\sqrt[4]{\mathfrak{I}(q)}\right).
				\end{align}
				\item 	
				\begin{align}\label{1O1+3O3}
					\nonumber \beta_1 	\Gamma_1^2(q^{1/2})  \Gamma_3(q^{1/2}) &+ \beta_3	\Gamma_1(q^{1/2})  \Gamma_3^2(q^{1/2}) \\ &= \dfrac{4 q^{-3/32} \eta^{1/2}(4\tau) \eta^{1/4}(\tau) \eta(8\tau)}{\eta(\tau/2) \eta^{3/4}(2\tau)} \sqrt{\mathfrak{H}(q)} \sqrt[4]{\mathfrak{I}(q)}.
				\end{align}
			\end{enumerate}
			Here, $\mathfrak{H}(q)$ and $\mathfrak{I}(q)$ are Ramanujan-G\"ollnitz-Gordon continued fraction  and the continued fraction of order four given by \eqref{H} and \eqref{I} respectively. 
		\end{theorem}
		\begin{proof}[Proof of \eqref{O1-O3}]
			Consider 
			\begin{align*}
				\Gamma_1(q) - \Gamma_3(q) &= \prod_{k=1}^{\infty} (1+\beta_1 q^k +q^{2k}) - \prod_{k=1}^{\infty} (1+\beta_3 q^k +q^{2k})\\
				&= \dfrac{q^{-1/12}}{\eta(\tau)} \left(\dfrac{\theta_1\left(\frac{\pi}{8}|\tau\right)}{2 \sin \left(\frac{\pi}{8}\right)}-\dfrac{\theta_1\left(\frac{3\pi}{8}|\tau\right)}{2 \sin \left(\frac{3\pi}{8}\right)}\right).
			\end{align*}
			Using the identity \eqref{t1}, the above identity can be written as
			\begin{align}\label{3.1.0}
				\Gamma_1(q) - \Gamma_3(q) 
				&= \dfrac{q^{-1/12}}{\eta(\tau)} \sum_{k=0}^{\infty} (-1)^k \mathfrak{B}_1(k) q^{\frac{(2k+1)^2}{8}},
			\end{align} where
			\begin{align*}
				\mathfrak{B}_1(k)= \dfrac{\sin (2k+1) \frac{\pi}{8}}{\sin \frac{\pi}{8}}- \dfrac{\sin (2k+1) \frac{3\pi}{8}}{\sin \frac{3\pi}{8}}.
			\end{align*}
			Using MAPLE computations, one can see that
			\begin{align*}
				\mathfrak{B}_1(8l+0) =0,  &\qquad
				\mathfrak{B}_1(8l+1)=2\sqrt{2}, \\
				\mathfrak{B}_1(8l+2)=2\sqrt{2}, &\qquad
				\mathfrak{B}_1(8l+3)=0, \\
				\mathfrak{B}_1(8l+4)=0, &\qquad
				\mathfrak{B}_1(8l+5)=-2\sqrt{2}, \\
				\mathfrak{B}_1(8l+6)=-2\sqrt{2}, &\qquad
				\mathfrak{B}_1(8l+7)=0.
			\end{align*}
			Thus,
			\begin{align*}
				\sum_{k=0}^{\infty} (-1)^k \mathfrak{B}_1(k) q^{\frac{(2k+1)^2}{8}} =&  2\sqrt{2}\left\{-\sum_{l=0}^{\infty} q^{\frac{(16l+3)^2}{8}} 
				+\sum_{l=0}^{\infty}  q^{\frac{(16l+5)^2}{8}}+\sum_{l=0}^{\infty}  q^{\frac{(16l+11)^2}{8}}\right. \\
				&\left.-\sum_{l=0}^{\infty}  q^{\frac{(16l+13)^2}{8}}\right\}.
			\end{align*}
			Changing $l$ to $-l$ and $l$ to $l-1$ in the first and last two summations respectively, above equation deduces to 
			
			\begin{align*}
				\sum_{k=0}^{\infty} (-1)^k \mathfrak{B}_1(k) q^{\frac{(2k+1)^2}{8}} =  2\sqrt{2}\left\{-\sum_{l=-\infty}^{\infty} q^{\frac{(16l+3)^2}{8}} 
				+\sum_{l=-\infty}^{\infty}  q^{\frac{(16l+5)^2}{8}} \right\}.
			\end{align*}
			Using the definition of $f(\gamma, \delta)$ in the above, the above expression reduces to 
			\begin{align}\label{3.1.1}
				\sum_{k=0}^{\infty} (-1)^k \mathfrak{B}_1(k) q^{\frac{(2k+1)^2}{8}} = 2\sqrt{2}q^{9/8} \left\{ q^2f(q^{12},q^{52})- f(q^{20}, q^{44}) \right\}.
			\end{align}
			Setting $(\gamma,\delta)=(q^2,q^{14})$ in \eqref{f4} yields   the following equation:
			\begin{align}\label{fab1}
				f(-q^2,-q^{14}) &= f(q^{20}, q^{44}) -q^2 f(q^{12}, q^{52}).
			\end{align}
			Using these in \eqref{3.1.1}, 
			\begin{align}\label{3.1.2}
				\sum_{k=0}^{\infty} (-1)^k \mathfrak{B}_1(k) q^{\frac{(2k+1)^2}{8}} = 2\sqrt{2}q^{9/8} f(-q^2,-q^{14}).
			\end{align}
			Utilizing \eqref{3.1.2} in \eqref{3.1.0}, one can see that
			\begin{align}\label{3.1.3}
				\Gamma_1(q) - \Gamma_3(q) 
				= \dfrac{2\sqrt{2}q^{25/24}}{\eta(\tau)} f(-q^2,-q^{14}).
			\end{align}
			Multiplying on both the sides of \eqref{3.1.3} by \eqref{prodK},
			\begin{align}\label{3.1.4}
				\Gamma_1^2(q)\Gamma_2(q) \Gamma_3(q) - \Gamma_1(q) \Gamma_2(q) \Gamma_3^2(q)
				= \dfrac{2\sqrt{2}q^{19/24}\eta(8\tau)}{\eta(\tau)\eta(2\tau)}f(-q^2,-q^{14}).
			\end{align}
			Utilizing the fact that $\beta_2=0$ on the left hand side of \eqref{3.1.4} and noticing that $(-q^2;q^2)_\infty = q^{-2/24}\dfrac{\eta(4\tau)}{\eta(2\tau)}$, upon simplification, changing $q$ to $q^{1/2}$ throughout and  using the definitions of $\mathfrak{H}$, $\mathfrak{I}$ and $\eta$, identity \eqref{O1-O3} is established.
			The proof of the identity \eqref{O1+O3} is analogous to the above and hence is omitted. 
		\end{proof}
		
		\begin{proof}[Proof of \eqref{33O3-11O1}]
			Consider 
			\begin{align*}
				(1+\beta_3)	\Gamma_3(q) &-(1+\beta_1) \Gamma_1(q)\\ &= (1+\beta_3)\prod_{k=1}^{\infty} (1+\beta_3 q^k +q^{2k}) - (1+\beta_1)\prod_{k=1}^{\infty} (1+\beta_1 q^k +q^{2k})\\
				&= \dfrac{q^{-1/12}}{\eta(\tau)} \left(\dfrac{(1+\beta_3)\theta_1\left(\frac{3\pi}{8}|\tau\right)}{2 \sin \left(\frac{3\pi}{8}\right)}-\dfrac{(1+\beta_1)\theta_1\left(\frac{\pi}{8}|\tau\right)}{2 \sin \left(\frac{\pi}{8}\right)}\right).
			\end{align*}
			Using the identity \eqref{t1}, the above identity can be written as
			\begin{align}\label{3.1.5}
				(1+\beta_3)	\Gamma_3(q) -(1+\beta_1) \Gamma_1(q) &=
				\dfrac{q^{-1/12}}{\eta(\tau)} \sum_{k=0}^{\infty} (-1)^k \mathfrak{B}_2(k) q^{\frac{(2k+1)^2}{8}},
			\end{align} where
			\begin{align*}
				\mathfrak{B}_2(k)= \dfrac{(1+\beta_3)\sin (2k+1) \frac{3\pi}{8}}{\sin \frac{3\pi}{8}}- \dfrac{(1+\beta_1)\sin (2k+1) \frac{\pi}{8}}{\sin \frac{\pi}{8}}.
			\end{align*}
			Using MAPLE computations, one can see that
			\begin{align*}
				\mathfrak{B}_2(8l+0) =2\sqrt{2},  &\qquad
				\mathfrak{B}_2(8l+1)=0, \\
				\mathfrak{B}_2(8l+2)=0, &\qquad
				\mathfrak{B}_2(8l+3)=2\sqrt{2}, \\
				\mathfrak{B}_2(8l+4)=-2\sqrt{2}, &\qquad
				\mathfrak{B}_2(8l+5)=0, \\
				\mathfrak{B}_2(8l+6)=0, &\qquad
				\mathfrak{B}_2(8l+7)=-2\sqrt{2}.
			\end{align*}
			Thus,
			\begin{align*}
				\sum_{k=0}^{\infty} (-1)^k \mathfrak{B}_2(k) q^{\frac{(2k+1)^2}{8}} =&  2\sqrt{2}\left\{\sum_{l=0}^{\infty} q^{\frac{(16l+1)^2}{8}} 
				-\sum_{l=0}^{\infty}  q^{\frac{(16l+7)^2}{8}}-\sum_{l=0}^{\infty}  q^{\frac{(16l+9)^2}{8}}\right. \\
				&\left.+\sum_{l=0}^{\infty}  q^{\frac{(16l+15)^2}{8}}\right\}.
			\end{align*}
			Changing $l$ to $-l$ and $l$ to $l-1$ in the first and last two summations respectively, above equation deduces to 
			
			\begin{align*}
				\sum_{k=0}^{\infty} (-1)^k \mathfrak{B}_2(k) q^{\frac{(2k+1)^2}{8}} =  2\sqrt{2}\left\{\sum_{l=-\infty}^{\infty} q^{\frac{(16l-1)^2}{8}} 
				-\sum_{l=-\infty}^{\infty}  q^{\frac{(16l-7)^2}{8}} \right\}.
			\end{align*}
			Using the definition of $f(\gamma,\delta)$, the above expression reduces to 
			\begin{align}\label{3.1.6}
				\sum_{k=0}^{\infty} (-1)^k \mathfrak{B}_2(k) q^{\frac{(2k+1)^2}{8}} = 2\sqrt{2}q^{1/8} \left\{  f(q^{28},q^{36})- q^6f(q^{4}, q^{60}) \right\}.
			\end{align}
			Setting $(\gamma,\delta)=(q^6,q^{10})$ in \eqref{f4} yields   the following equation:
			\begin{align}\label{fab2}
				f(-q^6,-q^{10}) &= f(q^{28}, q^{36}) -q^6 f(q^{4}, q^{60}).
			\end{align}
			Using these in \eqref{3.1.6}, 
			\begin{align}\label{3.1.7}
				\sum_{k=0}^{\infty} (-1)^k \mathfrak{B}_2(k) q^{\frac{(2k+1)^2}{8}} = 2\sqrt{2}q^{1/8} f(-q^6,-q^{10}).
			\end{align}
			Utilizing \eqref{3.1.7} in \eqref{3.1.5}, one can see that
			\begin{align}\label{3.1.8}
				(1+\beta_3)	\Gamma_3(q) -(1+\beta_1) \Gamma_1(q) 
				= \dfrac{2\sqrt{2}q^{1/24}}{\eta(\tau)} f(-q^6,-q^{10}).
			\end{align}
			Multiplying on both the sides of \eqref{3.1.8} by \eqref{prodK},
			\begin{align}\label{3.1.9}
				\nonumber	(1+\beta_3)	\Gamma_1(q)\Gamma_2(q) \Gamma_3^2(q) &- (1+\beta_1)\Gamma_1^2(q) \Gamma_2(q) \Gamma_3(q)\\
				&= \dfrac{2\sqrt{2}q^{-5/24}\eta(8\tau)}{\eta(\tau)\eta(2\tau)}f(-q^6,-q^{10}).
			\end{align}
			Utilizing the fact that $\beta_2=0$ on the left hand side of \eqref{3.1.9} and noticing that $(-q^2;q^2)_\infty = q^{-2/24}\dfrac{\eta(4\tau)}{\eta(2\tau)}$, upon simplification, changing $q$ to $q^{1/2}$ throughout and using the definitions of $\mathfrak{H}$, $\mathfrak{I}$ and $\eta$, identity \eqref{33O3-11O1} is established.
			The proof of the identity \eqref{11O1+33O3} is analogous to the above and hence is omitted.
		\end{proof}
		
		\begin{proof}[Proof of \eqref{3O3-1O1}]
			Consider 
			\begin{align*}
				\beta_3	\Gamma_3(q) &-\beta_1 \Gamma_1(q)\\ &=  \beta_3\prod_{k=1}^{\infty} (1+\beta_3 q^k +q^{2k}) - \beta_1\prod_{k=1}^{\infty} (1+\beta_1 q^k +q^{2k})\\
				&= \dfrac{q^{-1/12}}{\eta(\tau)} \left(\dfrac{\beta_3\theta_1\left(\frac{3\pi}{8}|\tau\right)}{2 \sin \left(\frac{3\pi}{8}\right)}-\dfrac{\beta_1\theta_1\left(\frac{\pi}{8}|\tau\right)}{2 \sin \left(\frac{\pi}{8}\right)}\right).
			\end{align*}
			Using the identity \eqref{t1}, the above identity can be written as
			\begin{align}\label{3.1.10}
				\beta_3	\Gamma_3(q) -\beta_1 \Gamma_1(q) &=
				\dfrac{q^{-1/12}}{\eta(\tau)} \sum_{k=0}^{\infty} (-1)^k \mathfrak{B}_3(k) q^{\frac{(2k+1)^2}{8}},
			\end{align} where
			\begin{align*}
				\mathfrak{B}_3(k)= \dfrac{\beta_3\sin (2k+1) \frac{3\pi}{8}}{\sin \frac{3\pi}{8}}- \dfrac{\beta_1\sin (2k+1) \frac{\pi}{8}}{\sin \frac{\pi}{8}}.
			\end{align*}
			Using MAPLE computations, one can see that
			\begin{align*}
				\mathfrak{B}_3(8l+0) =2\sqrt{2},  &\qquad
				\mathfrak{B}_3(8l+1)=2\sqrt{2}, \\
				\mathfrak{B}_3(8l+2)=2\sqrt{2}, &\qquad
				\mathfrak{B}_3(8l+3)=2\sqrt{2}, \\
				\mathfrak{B}_3(8l+4)=-2\sqrt{2}, &\qquad
				\mathfrak{B}_3(8l+5)=-2\sqrt{2}, \\
				\mathfrak{B}_3(8l+6)=-2\sqrt{2}, &\qquad
				\mathfrak{B}_3(8l+7)=-2\sqrt{2}.
			\end{align*}
			Thus,
			\begin{align*}
				&	\sum_{k=0}^{\infty} (-1)^k \mathfrak{B}_3(k) q^{\frac{(2k+1)^2}{8}} =  2\sqrt{2}\left\{\sum_{l=0}^{\infty} q^{\frac{(16l+1)^2}{8}} 
				-\sum_{l=0}^{\infty} q^{\frac{(16l+3)^2}{8}}+\sum_{l=0}^{\infty} q^{\frac{(16l+5)^2}{8}}\right. \\
				&\left.-\sum_{l=0}^{\infty}  q^{\frac{(16l+7)^2}{8}}-\sum_{l=0}^{\infty}  q^{\frac{(16l+9)^2}{8}}+\sum_{l=0}^{\infty} q^{\frac{(16l+11)^2}{8}}-\sum_{l=0}^{\infty} q^{\frac{(16l+13)^2}{8}}+\sum_{l=0}^{\infty}  q^{\frac{(16l+15)^2}{8}}\right\}.
			\end{align*}
			Changing $l$ to $-l$ and $l$ to $l-1$ in the first and last four summations respectively, above equation deduces to 
			
			\begin{align*}
				\sum_{k=0}^{\infty} (-1)^k \mathfrak{B}_3(k) q^{\frac{(2k+1)^2}{8}} &=  2\sqrt{2}\left\{\sum_{l=-\infty}^{\infty} q^{\frac{(16l-1)^2}{8}} 
				-\sum_{l=-\infty}^{\infty} q^{\frac{(16l-3)^2}{8}}+\sum_{l=-\infty}^{\infty} q^{\frac{(16l-5)^2}{8}}\right. \\
				&\left.-\sum_{l=-\infty}^{\infty}  q^{\frac{(16l-7)^2}{8}} \right\}.
			\end{align*}
			Using the definition of $f(\gamma,\delta)$, the above expression reduces to 
			\begin{align}\label{3.1.11}
				\nonumber	\sum_{k=0}^{\infty} (-1)^k \mathfrak{B}_3(k) q^{\frac{(2k+1)^2}{8}} &= 2\sqrt{2}q^{1/8} \left\{  f(q^{28},q^{36})-qf(q^{20},q^{44}) \right. \\ & \left. +q^3f(q^{12},q^{52})- q^6f(q^{4}, q^{60}) \right\}.
			\end{align}
			Using \eqref{fab1} and \eqref{fab2} in \eqref{3.1.11}, 
			\begin{align}\label{3.1.12}
				\sum_{k=0}^{\infty} (-1)^k \mathfrak{B}_3(k) q^{\frac{(2k+1)^2}{8}} = 2\sqrt{2}q^{1/8} \left(f(-q^6,-q^{10})-qf(-q^2,-q^{14})\right).
			\end{align}
			Utilizing \eqref{3.1.12} in \eqref{3.1.10}, one can see that
			\begin{align}\label{3.1.13}
				\beta_3	\Gamma_3(q) -\beta_1\Gamma_1(q) 
				= \dfrac{2\sqrt{2}q^{1/24}}{\eta(\tau)} \left(f(-q^6,-q^{10})-qf(-q^2,-q^{14})\right).
			\end{align}
			Multiplying on both the sides of \eqref{3.1.12} by \eqref{prodK},
			\begin{align}\label{3.1.14}
				\nonumber	\beta_3	\Gamma_1(q)\Gamma_2(q) \Gamma_3^2(q) &- \beta_1\Gamma_1^2(q) \Gamma_2(q) \Gamma_3(q)\\
				&= \dfrac{2\sqrt{2}q^{-5/24}\eta(8\tau)}{\eta(\tau)\eta(2\tau)}\left(f(-q^6,-q^{10})-qf(-q^2,-q^{14})\right).
			\end{align}
			Utilizing the fact that $\beta_2=0$ on the left hand side of \eqref{3.1.14} and noticing that $(-q^2;q^2)_\infty = q^{-2/24}\dfrac{\eta(4\tau)}{\eta(2\tau)}$, upon simplification, changing $q$ to $q^{1/2}$ throughout and using the definitions of $\mathfrak{H}$, $\mathfrak{I}$ and $\eta$, identity \eqref{3O3-1O1} is established.
			The proof of the identity \eqref{1O1+3O3} is analogous to the above and hence is omitted.
		\end{proof}

		\section{Lambert series identities using Ramanujan's $_1\psi_1$ summation formula}
		The theory of basic hypergeometric series facilitates a natural framework for studying Ramanujan's continued fractions and related $q$-series identities. Classical summation and transformation formulas reveal deep connections between continued fractions, theta functions, Lambert series, and modular forms. Ramanujan's bilateral summation formula $_1 \psi _1$ acquires a prominent role because of its application in deriving Lambert series expansions and theta-function identities.
		Ramanujan's continued fractions of various orders exhibit rich arithmetic and analytic structures, many of which can be understood through hypergeometric techniques. In this section, we explore new connections between the continued fraction of order twenty and the Lambert series. 
		We use Ramanujan's $_1 \psi _1$ summation formula \cite[Entry 17, p. 32]{Adiga_1985} given by 
		\begin{align}\label{e1}
			\sum_{j=-\infty}^{\infty} \dfrac{z^j}{1-xq^j} = \dfrac{(xz,q/xz,q,q;q)_\infty}{(x,q/x,z,q/z;q)_\infty}, \qquad |q|<|z|<1,
		\end{align} to establish Lambert series identities.
		\begin{theorem}
			For $|q|<1$, the following identity holds.
			\begin{align}\label{E1}
				\sum_{\substack{m=1 \\ m\equiv 1 \pmod 2}}^{\infty} \dfrac{q^{m} +q^{3m}}{1-q^{8m}} - \sum_{\substack{m=1 \\ m\equiv 1 \pmod 2}}^{\infty} \dfrac{q^{5m} +q^{7m}}{1-q^{8m}} = \dfrac{\eta^4(16\tau)}{\eta^2(8\tau)} \left[\dfrac{1}{\mathfrak{H}(q^2)}+\mathfrak{H}(q^2)\right].
			\end{align}
		\end{theorem}
		
		\begin{proof}
			Replacing $m$ by $-m$ in the second summation of the left-hand side of the identity \eqref{E1}, one can see that 
			\begin{align*}
				\sum_{\substack{m=1 \\ m\equiv 1 \pmod 2}}^{\infty} \dfrac{q^{m} +q^{3m}}{1-q^{8m}} &- \sum_{\substack{m=-\infty \\ m\equiv 1 \pmod 2}}^{-1} \dfrac{q^{-5m} +q^{-7m}}{1-q^{-8m}} \\ &= \sum_{m=-\infty}^{\infty} \dfrac{q^{2m+1}}{1-q^{16m+8}} +\sum_{m=-\infty}^{\infty} \dfrac{q^{6m+3}}{1-q^{16m+8}}.
			\end{align*}
			Utilizing \eqref{e1} in the above,	the following is deduced.
			\begin{align}\label{e2}
				\nonumber	\sum_{\substack{m=1 \\ m\equiv 1 \pmod 2}}^{\infty} &\dfrac{q^{m} +q^{3m}}{1-q^{8m}} - \sum_{\substack{m=-\infty \\ m\equiv 1 \pmod 2}}^{-1} \dfrac{q^{-5m} +q^{-7m}}{1-q^{-8m}} \\&= \dfrac{(q^{16};q^{16})_\infty^2}{(q^{8};q^{16})_\infty^2} \left[q\dfrac{ (q^6,q^{10};q^{16})_\infty}{(q^2,q^{14};q^{16})_\infty}+q^3\dfrac{ (q^2,q^{14};q^{16})_\infty}{(q^6,q^{10};q^{16})_\infty}\right].
			\end{align}
			Using the definitions of  $\mathfrak{H}$ and $\eta$  in \eqref{e2}, identity \eqref{E1} is obtained.  
		\end{proof}
		The following result is evident using \eqref{H+H} in the above.
		\begin{corollary}
			For $|q|<1$, the following identity holds.
			\begin{align}\label{E2}
				\sum_{\substack{m=1 \\ m\equiv 1 \pmod 2}}^{\infty} \dfrac{q^{m} +q^{3m}}{1-q^{8m}} - \sum_{\substack{m=1 \\ m\equiv 1 \pmod 2}}^{\infty} \dfrac{q^{5m} +q^{7m}}{1-q^{8m}} = \dfrac{\eta^4(16\tau)}{\eta^2(8\tau)} \left[\dfrac{\varphi(q^2)}{q\psi(q^8)}\right].
			\end{align}
		\end{corollary}
		
		\begin{theorem}
			For $|q|<1$, the following identities hold:
			\begin{enumerate}[(i)]
				\item 
				\begin{align}
					\sum_{\substack{m=1 \\ m\equiv 1 \pmod 2}}^{\infty} \dfrac{q^{m} -q^{3m}}{1-q^{8m}} + \sum_{\substack{m=1 \\ m\equiv 1 \pmod 2}}^{\infty} \dfrac{q^{5m} -q^{7m}}{1-q^{8m}} = \dfrac{\eta^4(16\tau)}{\eta^2(8\tau)} \left[\dfrac{\varphi(q^4)}{q \psi(q^8)}\right].
				\end{align}
				\item
				\begin{align}
					\sum_{\substack{m=1 \\ m\equiv 1 \pmod 4}}^{\infty} \dfrac{q^{m} +q^{5m}}{1-q^{8m}} - \sum_{\substack{m=1 \\ m\equiv 3 \pmod 4}}^{\infty} \dfrac{q^{3m} +q^{7m}}{1-q^{8m}} = \dfrac{\eta^2(32\tau)\eta(16\tau)}{\eta(8\tau)} \left[\dfrac{\varphi(q^4)}{q^{2} \psi(q^{16})}\right].
				\end{align}	
				\item 
				\begin{align}
					\nonumber	\sum_{\substack{m=1 \\ m\equiv 1 \pmod 8}}^{\infty} \dfrac{q^{m} -q^{5m}}{1-q^{8m}} &- \sum_{\substack{m=1 \\ m\equiv 3 \pmod 8}}^{\infty} \dfrac{q^{3m} +q^{7m}}{1-q^{8m}} + \sum_{\substack{m=1 \\ m\equiv 5 \pmod 8}}^{\infty} \dfrac{q^{m} +q^{5m}}{1-q^{8m}} \\&- \sum_{\substack{m=1 \\ m\equiv 7 \pmod 8}}^{\infty} \dfrac{q^{7m} -q^{3m}}{1-q^{8m}} = \dfrac{\eta^2(16\tau)\eta^2(64\tau)}{\eta(8\tau) \eta(32\tau)} \left[\dfrac{\varphi(q^{16})}{q^{4} \psi(q^{32})}\right].
				\end{align}
			\end{enumerate}
		\end{theorem}
		\begin{proof}
			The above three identities can be obtained by using \eqref{H-H} and \eqref{H+H} in Theorem 3.1 of the article \cite{Chaudhary_2025}.
		\end{proof}
		The left-hand sides of the above identities can be viewed as the generalized Lambert series associated with the residue classes modulo 20.
		For instance, let 
		\begin{align*}
			F_1(q):= \sum_{\substack{n\geq1 \\ n\equiv 1 \pmod 2}} \dfrac{q^{n}+q^{3n}-q^{5n}-q^{7n}}{1-q^{8n}}
		\end{align*} 
		$F_1(q)$ admits the Fourier series expansion:
		\begin{align*}
			F_1(q)= \sum_{N\geq 1} a(N) q^N
		\end{align*} 
		where 
		\begin{align*}
			a(N) = \sum_{\substack{d|N \\ d \,\, odd}} \chi(d),
		\end{align*}
		with 
		\begin{align*}
			\chi(d)=
			\begin{cases}
				1, & d\equiv 1,3 \pmod{8},\\
				-1, & d\equiv 5,7 \pmod{8},\\
				0, & \text{otherwise}.
			\end{cases}
		\end{align*}
		Note that 
		\begin{align*}
			q \dfrac{d }{dq} F_1(q) = \sum_{N\geq 1} N a(N) q^N.
		\end{align*} is a divisor weighted Lambert series of Eisenstein type.
		Thus, the logarithmic derivatives of $F_1(q)$  yields  divisor-weighted arithmetic functions having coefficients analogous to Eisenstein coefficients.
		
		\begin{theorem}
			For $|q|<1$, the following identities hold:
			\begin{enumerate}[(i)]
				\item 
				\begin{align}
					\sum_{m=1}^{\infty} \dfrac{m(q^{m} -q^{3m}-q^{5m}+q^{7m})}{1-q^{8m}} = \dfrac{\eta^4(8\tau)\eta^2(4\tau)}{\eta^2(2\tau)} \left[\dfrac{\varphi(q)}{q^{1/2} \psi(q^4)}\right].
				\end{align}
				\item
				\begin{align}
					\sum_{m=1}^{\infty} \left(\dfrac{m}{3}\right) \dfrac{(q^{m} -q^{3m}-q^{5m}+q^{7m})}{1-q^{8m}} = \dfrac{\eta(\tau) \eta^2(2\tau) \eta(6\tau)\eta^3(8\tau)\eta(24\tau)}{\eta(3\tau) \eta^5(4\tau)} \left[\dfrac{\varphi(q^2)}{q^{1/2} \psi(q^4)}\right].
				\end{align}	
			\end{enumerate}
			Here $\left(\dfrac{.}{p}\right)$ denotes the Legendre symbol modulo $p$.
		\end{theorem}
		\begin{proof}
			The above two identities can be obtained by using \eqref{H-H} and \eqref{H+H} in Theorem 3.3 of the article \cite{Chaudhary_2025}.
		\end{proof}
		
		\section{Conclusion}
		In this article, we established several identities connecting Ramanujan-G\"ollnitz-Gordon  continued fraction and Ramanujan's continued fraction of order four, using the product expansion of Jacobi's theta function $\theta_1$. The technique involved in this work suggests a systematic framework for deriving identities connecting continued fractions through theta function evaluations. Also, we obtained Lambert series identities utilizing  Ramanujan's $_1\psi_1$ summation formula.

		\section*{Acknowledgments}
		The first author acknowledges the support of DST-INSPIRE, Department of Science and Technology, Government of India, India for providing INSPIRE fellowship [DST/INSPIRE/03/2022/004970].

\end{document}